\documentclass[a4paper,12pt]{amsart}
\usepackage[utf8]{inputenc}
\usepackage{amsmath, amssymb, amsthm}
\usepackage{geometry}
\usepackage{enumitem}
\usepackage{hyperref}
\usepackage{booktabs}

\usepackage{listings}
\usepackage{xcolor}

\lstdefinestyle{python}{
  language=Python,
  basicstyle=\ttfamily\small,
  keywordstyle=\color{blue},
  commentstyle=\color{gray},
  stringstyle=\color{red!70!black},
  numbers=left,
  numberstyle=\tiny\color{gray},
  showstringspaces=false,
  breaklines=true,
  frame=single,
  tabsize=2,
  captionpos=b
}

\newtheorem{theorem}{Theorem}[section]
\newtheorem{lemma}[theorem]{Lemma}
\newtheorem{proposition}[theorem]{Proposition}
\newtheorem{corollary}[theorem]{Corollary}
\newtheorem{definition}[theorem]{Definition}

\title[Pillai's equation between $k$-Pell and Tribonacci numbers]{Pillai's equation between $k$-generalized Pell numbers and Tribonacci numbers}

\author[C. Deme]{Cherif B. Deme}
\address{C. Deme, UFR SAT, Universit\'e Alioune Diop, Bambey, 30, S\'en\'egal}
\email{cherifbachir.deme@uadb.edu.sn}

\author[K. D. Fall]{Kancou D. Fall}
\address{K. D. Fall,
UFR SAT, Universit\'e Alioune Diop, Bambey, 30, S\'en\'egal}
\email{kancou.d.fall@aims-senegal.org}

\author[B. Faye]{Bernadette Faye}
\address{B. Faye,
	UFR SAT, Universit\'e Alioune Diop, Bambey, 30, S\'en\'egal}
\email{bernadette.faye@uadb.edu.sn}

\begin{document}
\maketitle

\begin{abstract}
For an integer $k\ge2$, let $(P_n^{(k)})_{n\ge0}$ be the $k$-generalized Pell sequence and
$(T_m)_{m\ge1}$ the Tribonacci sequence. We determine, for every $k\ge2$, all integers $c$
having at least two representations $c=P_n^{(k)}-T_m$ with $m\ge2$. The starting point is an
identity of K{\i}l{\i}\c{c}: $P_n^{(k)}=F_{2n+1}$ for $0\le n\le k$. In the range $n\le k$
the problem therefore becomes, for all $k$ at once, Pillai's equation between odd-indexed
Fibonacci numbers and Tribonacci numbers, which has exactly six solutions. In the range $n>k$
we solve the problem for $k=2$ and $k=3$ using the exact heights of the Binet coefficients;
for $k\ge4$ we prove a bound for $n$ polynomial in $k$. We also use the fact that $P_n^{(k)}$ is
extremely close to $F_{2n+1}$ when $k$ is large to show that $k\le414$, and treat the
remaining values one by one. For $k\ge4$ there are, besides the solutions common to all $k$,
exactly four further solutions, for $k=4$, $5$ and $9$. The proofs combine Matveev's theorem
with the reduction method of Dujella and Pethő.
\end{abstract}

%

\section{Introduction}

For a fixed integer $c$, Pillai's equation asks for all pairs of representations
\[
c=x-y=x_1-y_1,\qquad (x,y)\ne(x_1,y_1),
\]
where $x,x_1$ range over one sequence and $y,y_1$ over another. The problem originates with
Pillai's 1936--1937 study \cite{Pillai1936,Pillai1937} of $a^x-b^y=c$ for fixed bases $a,b$. In the past few years, using the development of explicit lower bounds for
linear forms in complex logarithms of algebraic numbers, the same question has been studied for
many pairs of linear recurrence sequences in place of $a^x,b^y$: Fibonacci numbers against
powers of $2$ or $3$, Tribonacci numbers against powers of $2$ \cite{BLY2017}, and Padovan
numbers against Fibonacci numbers \cite{TA2022}.

This paper is about one such pair: the $k$-generalized Pell numbers $P_n^{(k)}$ against the Tribonacci numbers $T_m$. Unlike
the papers cited above, $k$ is not fixed at the outset, and this is what makes the pair
interesting. A first computation already shows something unusual which is that for every $k\ge3$ the
integer $c=0$ has three representations,
\[
0=P_0^{(k)}-T_2=P_1^{(k)}-T_3=P_3^{(k)}-T_6\qquad(1=1,\ 2=2,\ 13=13),
\]
and they do not depend on $k$. For $k=2$ the third one disappears, since $P_3^{(2)}=12$, and
the same indices give instead $c=-1=P_3^{(2)}-T_6=P_0^{(2)}-T_3$.\\
 The explanation is the
identity $$P_n^{(k)}=F_{2n+1},$$ for $0\le n\le k$, valid for every $k\ge2$, due to
K{\i}l{\i}\c{c} \cite{Kilic2008} stating that  the first $k+1$
terms of $P_n^{(k)}$ are exactly the odd-indexed Fibonacci numbers . As long as $n\le k$,
the problem does not see $k$ at all, and the natural questions are how far this independence
goes, and what happens once $n>k$. This splits Pillai's equation between $k$-Pell and
Tribonacci numbers into two ranges, treated in the two parts of this paper.

In Part~\ref{part:universal}, we solve the sub-problem restricted to $n\le k$: by the identity
above this is, for every $k$ at once, the same fixed Diophantine equation, Pillai's equation
between odd-indexed Fibonacci numbers and Tribonacci numbers, a pair that has not, to our
knowledge, been treated previously. Matveev's theorem gives an astronomically large bound,
$N<9.24\times10^{48}$; the reduction of Dujella and Pethő \cite{DP1998} brings it down to
$N\le71$, and a computer search leaves six solutions, all with $N\le7$, as in Tiebekabe and
Adonsou's treatment \cite{TA2022} of the Padovan--Fibonacci case. The structure of the proof
is summarized in Table~\ref{tab:roadmap}.

In Part~\ref{part:generalk}, we treat the complementary range $n>k$, where $P_n^{(k)}\ne
F_{2n+1}$. This is where the difficulty lies because the field $\mathbb Q(\gamma(k),\alpha)$ has
degree up to $3k$, and the height of the coefficient $f_k(\gamma)$ grows like $\log k$. For
$k=2$ and $k=3$ all the algebraic numbers involved can be computed exactly, and the method of
Part~\ref{part:universal} goes through. For $k\ge4$, Matveev's theorem only gives a bound
$n<\mathcal N(k)$ which grows with $k$ (Proposition~\ref{prop:uniform}), so the values of $k$
cannot simply be treated one at a time. For large $k$ we use the fact that the dominant root
$\gamma(k)$ is exponentially close to $\varphi^2$: then $P_n^{(k)}$ is extremely close to
$F_{2n+1}$ (Lemma~\ref{lem:approx}), the linear forms live in the field
$\mathbb Q(\sqrt5,\alpha)$ of degree $6$, independent of $k$, and the reduction shows that
$k\le414$ (Proposition~\ref{prop:largek}). The values $4\le k\le414$ are then treated one by
one. A second difficulty is that the argument used for $k=2,3$ to show that the linear forms
do not vanish, namely $\mathbb Q(\gamma)\cap\mathbb Q(\alpha)=\mathbb Q$, is not available for
every $k$; we replace it by an automorphism fixing $\alpha$ and moving $\gamma$.

Since $T_1=T_2=1$, every integer $c=P_n^{(k)}-1$ has the two trivial representations with
$m=1$ and $m=2$; as usual, we exclude them by requiring $m,m_1\ge2$.\\


Throughout, $(F_N)_{N\ge0}$ is the ordinary Fibonacci sequence, $F_0=0,F_1=1,
F_{N}=F_{N-1}+F_{N-2}$, and $(T_m)_{m\ge1}$ is the Tribonacci sequence, $T_1=T_2=1,T_3=2,
T_m=T_{m-1}+T_{m-2}+T_{m-3}$.

\begin{definition}\label{def:kpell}
For an integer $k\ge2$, the $k$-generalized Pell sequence $(P_n^{(k)})_{n\ge-(k-1)}$ is
defined by $P_{-(k-1)}^{(k)}=\cdots=P_{-1}^{(k)}=0$, $P_0^{(k)}=1$, and
\[
P_n^{(k)}=2P_{n-1}^{(k)}+P_{n-2}^{(k)}+\cdots+P_{n-k}^{(k)}\qquad(n\ge1).
\]
For $k=2$ this is the classical Pell sequence $1,2,5,12,29,\ldots$. Its characteristic
polynomial is $P_k(x)=x^k-2x^{k-1}-x^{k-2}-\cdots-x-1$.
\end{definition}

With this indexing, $P_n^{(k)}$ is the term denoted $P_{n+1}^{(k)}$ in
\cite{BHL2021,Kilic2008}. In all the statements below we write a solution of
\begin{equation}\label{eq:pillaiK}
c=P_n^{(k)}-T_m=P_{n_1}^{(k)}-T_{m_1}
\end{equation}
as $(n,m,n_1,m_1;c)$, normalized by $n>n_1\ge0$ and $m>m_1\ge2.$ 

\begin{theorem}[Main result, Part~\ref{part:universal}]\label{thm:main}
Let $k\ge2$. The solutions of \eqref{eq:pillaiK} with $n\le k$ are exactly
\[
(1,3,0,2;0),\qquad (2,4,1,2;1),\qquad (2,5,1,4;-2),
\]
and, when $k\ge3$, also
\[
(3,6,0,2;0),\qquad (3,6,1,3;0),\qquad (3,7,1,6;-11).
\]
In particular the list does not depend on $k$ once $k\ge3$.
\end{theorem}

Theorem~\ref{thm:main} is deduced (Section~\ref{sec:kpell-proof}) from K{\i}l{\i}\c{c}'s identity
(Lemmas~\ref{lem:prefix} and~\ref{lem:aisfib}) together with the following self-contained
result about ordinary Fibonacci and Tribonacci numbers, which is the technical heart of
Part~\ref{part:universal}.

\begin{theorem}[Key tool, Part~\ref{part:universal}]\label{thm:tool}
Let $N>N_1\ge1$ be odd integers and $m>m_1\ge2$ be integers such that
\begin{equation}\label{eq:pillai}
F_N-T_m=F_{N_1}-T_{m_1}.
\end{equation}
Then $(N,m,N_1,m_1)$ is one of
\[
(3,3,1,2),\quad(5,4,3,2),\quad(5,5,3,4),\quad(7,6,1,2),\quad(7,6,3,3),\quad(7,7,3,6).
\]
\end{theorem}

\begin{theorem}[Part~\ref{part:generalk}]\label{thm:k2}
For $k=2$ (classical Pell numbers), the solutions of \eqref{eq:pillaiK} are exactly
\[
\begin{gathered}
(1,3,0,2;0),\quad(2,4,1,2;1),\quad(2,5,1,4;-2),\quad(3,6,0,3;-1),\quad(3,7,0,6;-12),\\
(4,7,3,5;5),\quad(5,9,1,6;-11),\quad(5,10,1,9;-79),\quad(9,15,6,13;-758).
\end{gathered}
\]
\end{theorem}

\begin{theorem}[Part~\ref{part:generalk}]\label{thm:k3}
For $k=3$, the solutions of \eqref{eq:pillaiK} are exactly
\[
\begin{gathered}
(1,3,0,2;0),\quad(2,4,1,2;1),\quad(2,5,1,4;-2),\quad(3,6,0,2;0),\quad(3,6,1,3;0),\\
(3,7,1,6;-11),\quad(4,8,1,6;-11),\quad(4,7,3,4;9),\quad(4,8,3,7;-11),\quad(5,9,2,3;3).
\end{gathered}
\]
\end{theorem}

\begin{theorem}[Part~\ref{part:generalk}]\label{thm:kge4}
Let $k\ge4$. The solutions of \eqref{eq:pillaiK} with $n>k$ are exactly
\[
\begin{gathered}
(7,12,5,2;87)\ \text{for }k=4,\qquad (6,11,1,8;-42),\ (6,12,1,11;-272)\ \text{for }k=5,\\
(10,17,7,11;336)\ \text{for }k=9,
\end{gathered}
\]
and there is no such solution for the other values of $k\ge4$.
\end{theorem}

Theorems~\ref{thm:k2} and~\ref{thm:k3} are the complete lists for $k=2$ and $k=3$: they
contain the solutions with $n\le k$ given by Theorem~\ref{thm:main} (the first three solutions
for $k=2$, the first six for $k=3$), and the solutions with $n>k$, which are the ones obtained
in Part~\ref{part:generalk}. For $k\ge4$, Theorem~\ref{thm:main} (range $n\le k$) and
Theorem~\ref{thm:kge4} (range $n>k$) together give all the solutions of \eqref{eq:pillaiK}.

\begin{corollary}\label{cor:c}
The integers $c$ having at least two representations $c=P_n^{(k)}-T_m$ with $m\ge2$ are
\[
c\in\{-758,-79,-12,-11,-2,-1,0,1,5\}\quad(k=2),\qquad c\in\{-11,-2,0,1,3,9\}\quad(k=3).
\]
For $k=2$ each of them has exactly two representations; for $k=3$ the integers $c=0$ and
$c=-11$ have three representations and the others exactly two. For $k\ge4$ they are
\[
c\in\{-11,-2,0,1\}\cup E_k,\qquad E_4=\{87\},\quad E_5=\{-272,-42\},\quad E_9=\{336\},
\]
and $E_k=\emptyset$ otherwise;
here $c=0$ has three representations and the others exactly two.
\end{corollary}

The solutions with $c=0$ correspond to common values of the two sequences namely for $1=P_0^{(k)}=T_2$, $2=P_1^{(k)}=T_3$ and, for $k\ge3$, $13=P_3^{(k)}=T_6$.
\\


This paper is organized as follows: Section~\ref{sec:prelim} collects the preliminary properties shared by both parts, which are the Binet-type
formulas for Fibonacci and Tribonacci numbers, Matveev's theorem, the Dujella--Pethő
reduction lemma, and K{\i}l{\i}\c{c}'s identity, from which Theorem~\ref{thm:main} is deduced
(Section~\ref{sec:kpell-proof}). Part~\ref{part:universal} (Sections~\ref{sec:setup}--%
\ref{sec:reduction}) proves Theorem~\ref{thm:tool}; its structure is shown in
Table~\ref{tab:roadmap}. Part~\ref{part:generalk} treats the range $n>k$: $k=2,3$ in
Sections~\ref{sec:kpell-apparatus}--\ref{sec:reductionk}, the uniform bound for $k\ge4$ in
Section~\ref{sec:uniform}, large $k$ in Section~\ref{sec:largek} and $4\le k\le414$ in
Section~\ref{sec:smallk}. Appendix~\ref{app:comp} gives all the details of the computations,
which are checked by our two Python scripts.
\section{Preliminaries}\label{sec:prelim}

\subsection{Fibonacci numbers}

Let $\varphi=\frac{1+\sqrt5}2$ and $\psi=\frac{1-\sqrt5}2$ be the two roots of
$x^2-x-1$, so $|\psi|<1$. Binet's formula gives
\[
F_N=\frac{\varphi^N-\psi^N}{\sqrt5}\qquad(N\ge0),
\]
and since $|\psi|<1$,
\begin{equation}\label{eq:fib-binet}
\left|F_N-\frac{\varphi^N}{\sqrt5}\right|=\frac{|\psi|^N}{\sqrt5}<\frac1{\sqrt5}<0.4473
\qquad(N\ge0).
\end{equation}
We recall two elementary consequences used repeatedly below.

\begin{lemma}\label{lem:fib-elem}
For every $N\ge2$, $\ F_N<0.618\,\varphi^N$, and $F_1\le0.6181\,\varphi$. For every $N\ge2$,
$F_{N-1}\ge0.3819\,\varphi^{N-1}$.
\end{lemma}
\begin{proof}
Since $\psi=-\varphi^{-1}$, Binet's formula gives, for $j\ge1$,
\[
\frac{F_j}{\varphi^j}=\frac{1-(-\varphi^{-2})^{j}}{\sqrt5}.
\]
For odd $j$ this lies in $\big(1/\sqrt5,(1+\varphi^{-2j})/\sqrt5\big]$, and for even $j$ in
$\big[(1-\varphi^{-2j})/\sqrt5,1/\sqrt5\big)$. Hence, for $j\ge2$,
$$F_j/\varphi^j\le(1+\varphi^{-6})/\sqrt5<0.618$$ which is attained at $j=3$, while
$F_1/\varphi=\varphi^{-1}\le0.6181$; and for every $j\ge1$,
$F_j/\varphi^j\ge(1-\varphi^{-4})/\sqrt5=1/\varphi^2$, attained at $j=2$.
Taking $j=N$, respectively $j=N-1\ge1$, gives the statement.
\end{proof}

\begin{lemma}
\label{lem:oddsum}
For every $n\ge1$, $\displaystyle\sum_{i=1}^{n}F_{2i-1}=F_{2n}$.
\end{lemma}
\begin{proof}
The Fibonacci recurrence gives $F_{2i-1}=F_{2i}-F_{2i-2}$ for every $i\ge1$ (with $F_0=0$).
Summing this identity for $i=1,\ldots,n$ gives
$$\sum_{i=1}^nF_{2i-1}=\sum_{i=1}^n(F_{2i}-F_{2i-2})=F_{2n}-F_0=F_{2n}.$$
\end{proof}

\subsection{Tribonacci numbers}

Let $\alpha$ be the unique real root of $x^3-x^2-x-1$ of modulus greater than
$1$, and set
$$g(\alpha)=\dfrac{\alpha}{3\alpha^2-2\alpha-1}.$$ The following facts are
standard in the literature on Tribonacci Diophantine equations (see e.g.\ \cite{BLY2017});
\begin{equation}\label{eq:trib-binet}
T_m=g(\alpha)\alpha^m+e(m),\qquad |e(m)|<0.4\quad(m\ge1),
\end{equation}
\begin{equation}\label{eq:trib-heights}
h(\alpha)=\frac{\log\alpha}3,\qquad h(g(\alpha))=\frac{\log44}3.
\end{equation}
The second value comes from the minimal polynomial $44x^3-2x-1$ of $g(\alpha)$, whose three
roots have modulus less than $1$. Moreover,
\begin{equation}\label{eq:trib-size}
T_m<0.5436\,\alpha^m\quad(m\ge2),\qquad T_m\ge0.3214\,\alpha^m\quad(m\ge3).
\end{equation}
 Finally
$\mathbb Q(\alpha)\cap\mathbb Q(\sqrt5)=\mathbb Q$, the two fields having coprime degrees $3$
and $2$. We will use this result below to show that our linear forms do not vanish.

\subsection{Matveev's theorem}

\begin{theorem}[Matveev \cite{Matveev2000}]\label{thm:matveev}
Let $K$ be a number field of degree $D$ over $\mathbb Q$, let $\eta_1,\ldots,\eta_t\in K^*$
be positive real algebraic numbers, and let $b_1,\ldots,b_t$ be nonzero rational integers.
Put $\Lambda:=\eta_1^{b_1}\cdots\eta_t^{b_t}-1\ne0$. Let $A_1,\ldots,A_t$ be real numbers with
$A_i\ge\max\{Dh(\eta_i),|\log\eta_i|,0.16\}$ for each $i$, and $B\ge\max\{|b_1|,\ldots,
|b_t|\}$. Then
\[
\log|\Lambda|>-1.4\cdot30^{t+3}\cdot t^{4.5}\cdot D^2(1+\log D)(1+\log B)\,A_1\cdots A_t.
\]
\end{theorem}

The right-hand side increases with $D$, so the inequality remains true if $D$ is replaced by
an upper bound for $[K:\mathbb Q]$. For $t=3$ we write
\[
C(D):=1.4\cdot30^{6}\cdot3^{4.5}\cdot D^2(1+\log D).
\]

\subsection{The Dujella--Pethő reduction lemma}

\begin{lemma}[Dujella--Pethő \cite{DP1998}]\label{lem:dp}
Let $M$ be a positive integer, let $\theta$ be an irrational real number with convergent
$p/q$ of its continued fraction expansion satisfying $q>6M$, and let $A,B,\mu$ be real
numbers with $A>0$ and $B>1$. Put $\varepsilon:=\|\mu q\|-M\|\theta q\|$, where
$\|\cdot\|$ denotes the distance to the nearest integer. If $\varepsilon>0$, then there is
no solution of the inequality
\[
0<|u\theta-v+\mu|<AB^{-w}
\]
in integers $u,v,w$ with $1\le u\le M$ and $w\ge\log(Aq/\varepsilon)/\log B$.
\end{lemma}

Its proof works when $w$ is a real number, which is how we use it. In all our
applications $\varepsilon>0$ already holds for the first convergent with $q>6M$ or for one of
the next five convergents.

\subsection{The universal prefix lemma}

The identity proved in this subsection is due to K{\i}l{\i}\c{c} \cite{Kilic2008}.
We include a
short proof to keep the paper self-contained.

\begin{definition}
Let $(a_n)_{n\ge0}$ be the sequence with $a_0=1$ and $a_n=2a_{n-1}+a_{n-2}+\cdots+a_0$ for
$n\ge1$.
\end{definition}

\begin{lemma}[Universal prefix]\label{lem:prefix}
For every $k\ge2$ and every $0\le n\le k$, $\ P_n^{(k)}=a_n$.
\end{lemma}
\begin{proof}
We induct on $n$. For $n=0$ both sides equal $1$. Suppose $1\le n\le k$ and $P_j^{(k)}=a_j$
for all $0\le j<n$. By Definition~\ref{def:kpell},
$$P_n^{(k)}=2P_{n-1}^{(k)}+\sum_{i=2}^{k}P_{n-i}^{(k)}.$$
For $1\le i\le k$, either $n-i\ge0$ -- in which case $n-i<n\le k$, so $P_{n-i}^{(k)}=a_{n-i}$
by the induction hypothesis -- or $n-i<0$, in which case $n-i\ge n-k\ge1-k=-(k-1)$, so
$P_{n-i}^{(k)}=0$ by Definition~\ref{def:kpell}. Hence the sum collects exactly
$2a_{n-1}+a_{n-2}+\cdots+a_0$, the defining recursion for $a_n$. Thus $P_n^{(k)}=a_n$.
\end{proof}

\begin{lemma}\label{lem:aisfib}
For every $n\ge0$, $\ a_n=F_{2n+1}$.
\end{lemma}
\begin{proof}
We show $b_n:=F_{2n+1}$ satisfies the same recursive definition as $a_n$; since a sequence is
uniquely determined by such a recursion together with its initial value, this gives
$a_n=b_n$. First, $b_0=F_1=1=a_0$. Fix $n\ge1$; we must show
\begin{equation}\label{eq:toshow}
F_{2n+1}=2F_{2n-1}+\sum_{i=1}^{n-1}F_{2i-1}.
\end{equation}
This sum is empty when $n=1$. By Lemma~\ref{lem:oddsum}, $\sum_{i=1}^{n-1}F_{2i-1}=F_{2n-2}.$
This gives also $0$ when $n=1$, consistently with $F_0=0$. So \eqref{eq:toshow} reduces to
$F_{2n+1}=2F_{2n-1}+F_{2n-2}$, which holds by two applications of the Fibonacci recurrence,
$$F_{2n+1}=F_{2n}+F_{2n-1}=(F_{2n-1}+F_{2n-2})+F_{2n-1}=2F_{2n-1}+F_{2n-2}.$$
\end{proof}

Combining Lemmas~\ref{lem:prefix} and~\ref{lem:aisfib} gives, for every $k\ge2$ and every
$0\le n\le k$, the exact identity $P_n^{(k)}=F_{2n+1}$. 
%

\subsection{Proof of the main theorem 1.2}\label{sec:kpell-proof}
Suppose $$c=P_n^{(k)}-T_m=P_{n_1}^{(k)}-T_{m_1}$$ with $n>n_1\ge0$, $m>m_1\ge2$ and $n\le k$
(hence also $n_1\le k$). By Lemmas~\ref{lem:prefix} and~\ref{lem:aisfib}, $P_n^{(k)}=F_{2n+1}$
and $P_{n_1}^{(k)}=F_{2n_1+1}$. Writing $N:=2n+1$, $N_1:=2n_1+1$, the hypothesis becomes
$c=F_N-T_m=F_{N_1}-T_{m_1}$ with $N>N_1\ge1$ odd and $m>m_1\ge2$, exactly equation
\eqref{eq:pillai}. Conversely, every solution of \eqref{eq:pillai} with $(N-1)/2\le k$ gives a
solution of \eqref{eq:pillaiK} with $n\le k$. By Theorem~\ref{thm:tool}, the solutions are
the six listed there; writing $n=(N-1)/2$, $n_1=(N_1-1)/2$ gives the list of
Theorem~\ref{thm:main}, the last three solutions having $n=3$ and thus requiring $k\ge3$.

\part{The universal range $n\le k$}\label{part:universal}

This part proves Theorem~\ref{thm:tool}, the self-contained statement about ordinary
Fibonacci and Tribonacci numbers from which Theorem~\ref{thm:main} was deduced above. The
proof has five steps, summarized in Table~\ref{tab:roadmap}.

\begin{table}[ht]
\centering\footnotesize\setlength{\tabcolsep}{4pt}
\caption{Structure of the proof of Theorem~\ref{thm:tool}.}\label{tab:roadmap}
\begin{tabular}{@{}llll@{}}
\toprule
Section & Linear form & Upper bound & Conclusion\\
\midrule
\ref{sec:form1} & $\Lambda$ & $10\max(\varphi^{-\gamma_N},\alpha^{-\gamma_m})$ &
  one of the two gaps is $\ll\log N$\\
\ref{sec:form2} & $\Lambda_1$ (Case 1), $\Lambda_2$ (Case 2) &
  $5.5\,\alpha^{-\gamma_m}$, $3\,\varphi^{-\gamma_N}$ & both gaps are $\ll(\log N)^2$\\
\ref{sec:form4} & $\Lambda_3$ & $6.025\,\alpha^{-m_1}$ &
  $m_1\ll(\log N)^3$, hence $N<9.24\times10^{48}$\\
\ref{sec:reduction} & $\Lambda$, $\Lambda_1$/$\Lambda_2$, $\Lambda_3$ & &
  Dujella--Peth\H{o}, two passes: $N\le71$\\
\ref{sec:reduction} & & & computer search: six solutions\\
\bottomrule
\end{tabular}
\end{table}

\section{Setup}\label{sec:setup}

Suppose
\[
c=F_N-T_m=F_{N_1}-T_{m_1},\qquad N,N_1\ \text{odd},\ m,m_1\ge2,\ (N,m)\ne(N_1,m_1).
\]
If $N=N_1$, then $F_N=F_{N_1}$ forces $T_m=T_{m_1}$, hence $m=m_1$ (Tribonacci numbers are
strictly increasing for $m\ge2$), contradicting $(N,m)\ne(N_1,m_1)$. So $N\ne N_1$; without
loss of generality $N>N_1$. Thus, since both are odd, $N_1\le N-2$, and then
$F_N-F_{N_1}=T_m-T_{m_1}>0$ forces $T_m>T_{m_1}$, i.e.\ $m>m_1$. We write
\[
\gamma_N:=N-N_1\ge2\ ,\qquad \gamma_m:=m-m_1\ge1,
\]
so that
\begin{equation}\label{eq:main}
F_N-F_{N_1}=T_m-T_{m_1}.
\end{equation}
A direct computer search over all odd $N\le1501$ (Appendix~\ref{app:comp}) gives exactly the
six solutions listed in Theorem~\ref{thm:tool}, all with $N\le7$. Hence we assume from now that $N\ge15$.  By Lemma~\ref{lem:size} below, this gives $m\ge11$ and $m<N$, so that
$B=N$ can be used in Matveev's theorem.

Note that $N_1=1$ does occur, in $(3,3,1,2)$ and $(7,6,1,2)$, so the case $F_{N_1}=F_1=1$ has
to be covered by the estimates below; this is why Lemma~\ref{lem:fib-elem} has a separate
bound for $F_1$.

\section{Size-relation lemma}\label{sec:size}

\begin{lemma}\label{lem:size}
For $N\ge3$ and $m\ge3$ satisfying \eqref{eq:main}: $\ N<1.2666\,m+1.7336$; if moreover
$m\ge5$, then $\ m<0.7897\,N+3.073$. If $N\ge15$, then $m\ge11$ and $m<N$.
\end{lemma}
\begin{proof}
Since $N_1\le N-2$ and Fibonacci numbers are increasing,
$F_N-F_{N_1}\ge F_N-F_{N-2}=F_{N-1}\ge0.3819\,\varphi^{N-1}$  by Lemma~\ref{lem:fib-elem}. Again, 
$F_N-F_{N_1}<F_N<0.618\,\varphi^N$. Similarly, since $m_1\le m-1$ and Tribonacci numbers are
increasing, $T_m-T_{m_1}\ge T_m-T_{m-1}=T_{m-2}+T_{m-3}\ge T_{m-2}\ge0.3214\,\alpha^{m-2}$
for $m-2\ge3$, and $T_m<0.5436\,\alpha^m$, both by \eqref{eq:trib-size}. Combining the lower
bound on one side of \eqref{eq:main} with the upper bound on the other, in both directions,
and taking logarithms, gives
\[
N<1+\frac{\log\alpha}{\log\varphi}\,m+\frac{\log(0.5436/0.3819)}{\log\varphi}
<1.26634\,m+1.73368,
\]
and
\[
m<2+\frac{\log\varphi}{\log\alpha}\,N+\frac{\log(0.618/0.3214)}{\log\alpha}<0.78968\,N+3.0730,
\]
which implies the stated inequalities. If $N\ge15$, the first one gives $m>11$, , and the second one gives $m<0.7897N+3.073\le N$.
\end{proof}

We also note that the first chain of inequalities gives
\begin{equation}\label{eq:phiN}
\varphi^N<\varphi\cdot\frac{0.5436}{0.3819}\,\alpha^m<2.304\,\alpha^m.
\end{equation}

\section{First linear form and Matveev}\label{sec:form1}

Bounding $|c|$ via the small representation, and using $F_{N_1}\le0.6181\,\varphi^{N_1}$ and the Tribonacci analogue $T_{m_1}<0.5436\,\alpha^{m_1}$, we get
$$|c|<0.7\max(\varphi^{N_1},\alpha^{m_1}),$$ using Binet's formulas \eqref{eq:fib-binet},
\eqref{eq:trib-binet} on the large representation $c=F_N-T_m$,
\[
\left|\frac{\varphi^N}{\sqrt5}-g(\alpha)\alpha^m\right|<0.7\max(\varphi^{N_1},\alpha^{m_1})
+0.85<1.4\max(\varphi^{N_1},\alpha^{m_1}).
\]
Dividing by $g(\alpha)\alpha^m$ and setting
$\Lambda:=\frac1{\sqrt5\,g(\alpha)}\,\varphi^N\alpha^{-m}-1$, and using \eqref{eq:phiN}
\begin{equation}\label{eq:lambda}
|\Lambda|\ <\ 10\,\max\big(\varphi^{-\gamma_N},\ \alpha^{-\gamma_m}\big).
\end{equation}
We have $\Lambda\ne0$ because otherwise $\varphi^N/\sqrt5=g(\alpha)\alpha^m$ would belong to
$\mathbb Q(\sqrt5)\cap\mathbb Q(\alpha)=\mathbb Q$, whereas
$\varphi^N/\sqrt5=F_N/2+L_N/(2\sqrt5)$ is irrational, the Lucas number $L_N$ being nonzero.
The data for Matveev's Theorem~\ref{thm:matveev} for this form and for the three
forms of the next two sections are collected in Table~\ref{tab:matveev1}. For $\Lambda$,
$A_3=D\big(h(\sqrt5)+h(g(\alpha))\big)=12.397$, and comparing with \eqref{eq:lambda} we get:

\begin{table}[ht]
\centering\footnotesize\setlength{\tabcolsep}{4pt}
\caption{Data for Matveev's theorem in Part~\ref{part:universal}: $t=3$,
$K=\mathbb Q(\sqrt5,\alpha)$, $D=6$, $B=N$, $b_3=1$, $A_1=3\log\varphi=1.4436$,
$A_2=2\log\alpha=1.2188$, and $X:=1+\log N$.}\label{tab:matveev1}
\begin{tabular}{@{}lllll@{}}
\toprule
Form & $(\eta_1,b_1)$, $(\eta_2,b_2)$ & $\eta_3$ & $A_3$ & Result\\
\midrule
$\Lambda$ & $(\varphi,N)$, $(\alpha,-m)$ & $1/(\sqrt5\,g(\alpha))$ & $12.397$ &
  $\min\{\gamma_N\log\varphi,\gamma_m\log\alpha\}<K_1X$\\[3pt]
$\Lambda_1$ & $(\varphi,N_1)$, $(\alpha,-m)$ & $(\varphi^{\gamma_N}-1)/(\sqrt5\,g(\alpha))$ &
  $(3K_1+16.56)X$ & $\gamma_m\log\alpha<K_2X^2$\\[3pt]
$\Lambda_2$ & $(\varphi,-N)$, $(\alpha,m_1)$ & $\sqrt5\,g(\alpha)(\alpha^{\gamma_m}-1)$ &
  $(2K_1+16.56)X$ & $\gamma_N\log\varphi<K_3X^2$\\[3pt]
$\Lambda_3$ & $(\varphi,N_1)$, $(\alpha,-m_1)$ &
  $\dfrac{\varphi^{\gamma_N}-1}{\sqrt5\,g(\alpha)(\alpha^{\gamma_m}-1)}$ &
  $(5K_2+20.72)X^2$ & $m_1\log\alpha<K_4X^3$\\
\bottomrule
\end{tabular}
\end{table}

\begin{proposition}\label{prop:1}
For every solution of \eqref{eq:main} with $N\ge15$,
\[
\min\{\gamma_N\log\varphi,\ \gamma_m\log\alpha\}\ <\ K_1(1+\log N),\qquad K_1=3.139
\times10^{14}.
\]
\end{proposition}

Here $C(6)A_1A_2A_3>3.13879\times10^{14}$, and $K_1$ also absorbs the $\log10$
coming from \eqref{eq:lambda}.

\section{Secondary linear forms: bounding both gaps}\label{sec:form2}

\subsection*{Case 1: $\gamma_N\log\varphi\le\gamma_m\log\alpha$}

By Proposition~\ref{prop:1}, $\gamma_N<K_1(1+\log N)/\log\varphi$. Factoring $\varphi^{N_1}$
out of $F_N-F_{N_1}$ via Binet,
\[
F_N-F_{N_1}=\frac{\varphi^{N_1}(\varphi^{\gamma_N}-1)}{\sqrt5}+\delta_1,\qquad|\delta_1|<0.9,
\]
substituting $T_m=F_N-F_{N_1}+T_{m_1}$, dividing by $g(\alpha)\alpha^m$, and setting
$$\Lambda_1:=\frac{\varphi^{\gamma_N}-1}{\sqrt5g(\alpha)}\varphi^{N_1}\alpha^{-m}-1,$$ we get
\begin{equation}\label{eq:lambda1}
|\Lambda_1|\ <\ 5.5\,\alpha^{-\gamma_m}.
\end{equation}
We used the fact  $T_{m_1}<0.5436\,\alpha^{m-\gamma_m}$. We have
$\Lambda_1\ne0$, otherwise $(\varphi^N-\varphi^{N_1})/\sqrt5=g(\alpha)\alpha^m$ would be
rational, whereas $(\varphi^N-\varphi^{N_1})/\sqrt5=(F_N-F_{N_1})/2+(L_N-L_{N_1})/(2\sqrt5)$
and $L_N\ne L_{N_1}$, since the Lucas numbers are strictly increasing for $N\ge1$. Note
that $\Lambda_1$ contains $\alpha^{-m}$. In fact, for the Tribonacci sequence, the gap
$\gamma_m$ appears only in the upper bound \eqref{eq:lambda1}, and not in the form, so that
$b_2=-m$ in Table~\ref{tab:matveev1}. For the height of $\eta_3$,
\begin{align*}
D\,h(\eta_3)&\le D\,h(\varphi)\,\gamma_N+D\big(\log2+h(\sqrt5)+h(g(\alpha))\big)\\
&<3\gamma_N\log\varphi+16.56<(3K_1+16.56)(1+\log N)
\end{align*}
by the Case-1 hypothesis on $\gamma_N$, and Matveev's theorem ( see the second row of
Table~\ref{tab:matveev1}). This gives

\begin{proposition}\label{prop:2}
In Case 1, $\ \gamma_m\log\alpha\ <\ K_2(1+\log N)^2$, with $K_2=2.385\times10^{28}$.
\end{proposition}

\subsection*{Case 2: $\gamma_m\log\alpha\le\gamma_N\log\varphi$}

By the symmetric construction, factoring $\alpha^{m_1}$ out of $T_m-T_{m_1}$ and dividing by
$\varphi^N/\sqrt5$, with
\[
\Lambda_2:=\sqrt5\,g(\alpha)(\alpha^{\gamma_m}-1)\,\alpha^{m_1}\varphi^{-N}-1,
\]
we have $\varphi^N/\sqrt5-g(\alpha)\alpha^{m_1}(\alpha^{\gamma_m}-1)=F_{N_1}+\psi^N/\sqrt5+
e(m)-e(m_1)$, whose absolute value is less than $0.7\varphi^{N_1}+0.906<1.27\varphi^{N_1}$,
hence
\begin{equation}\label{eq:lambda2}
|\Lambda_2|\ <\ 3\,\varphi^{-\gamma_N}.
\end{equation}
As before $\Lambda_2\ne0$, since $\varphi^N/\sqrt5$ is irrational. Since
$D\,h(\eta_3)<2\gamma_m\log\alpha+16.56<(2K_1+16.56)(1+\log N)$, Matveev's theorem (third row
of Table~\ref{tab:matveev1}) gives:

\begin{proposition}\label{prop:3}
In Case 2, $\ \gamma_N\log\varphi\ <\ K_3(1+\log N)^2$, with $K_3=1.590\times10^{28}$.
\end{proposition}

Since $K_1(1+\log N)<K_3(1+\log N)^2<K_2(1+\log N)^2$, in both cases
\begin{equation}\label{eq:gaps}
\gamma_N\log\varphi<K_2(1+\log N)^2,\qquad \gamma_m\log\alpha<K_2(1+\log N)^2.
\end{equation}

\section{A fourth linear form: an absolute bound on $N$}\label{sec:form4}

Substituting Binet's formula into all four terms of $F_N-F_{N_1}=T_m-T_{m_1}$ and factoring both gaps
out,
\[
\left|\frac{\varphi^{N_1}(\varphi^{\gamma_N}-1)}{\sqrt5}-g(\alpha)\alpha^{m_1}
(\alpha^{\gamma_m}-1)\right|<1.7.
\]
Dividing by $g(\alpha)\alpha^{m_1}(\alpha^{\gamma_m}-1)\ge g(\alpha)(\alpha-1)\alpha^{m_1}$ and
setting
\[
\Lambda_3:=\frac{\varphi^{\gamma_N}-1}{\sqrt5\,g(\alpha)(\alpha^{\gamma_m}-1)}\,
\varphi^{N_1}\alpha^{-m_1}-1,\qquad|\Lambda_3|<6.025\,\alpha^{-m_1}.
\]
Again $\Lambda_3\ne0$ by the same argument given above for
$\Lambda_1$. The height of
$\eta_3$ is bounded via \eqref{eq:gaps} by
\[
D\,h(\eta_3)\le3\gamma_N\log\varphi+2\gamma_m\log\alpha+D\big(2\log2+h(\sqrt5)+h(g(\alpha))
\big)<(5K_2+20.72)(1+\log N)^2,
\]
and Matveev's theorem (last row of Table~\ref{tab:matveev1}) gives
\[
m_1\log\alpha<K_4(1+\log N)^3,\qquad K_4=3.020\times10^{42}.
\]
Combined with the linear-in-$N$ lower bound on $m_1=m-\gamma_m$ coming from
Lemma~\ref{lem:size} and \eqref{eq:gaps},
\[
\frac{N-1.7336}{1.2666}\,\log\alpha<m\log\alpha<K_4(1+\log N)^3+K_2(1+\log N)^2.
\]
This gives the following upper bound 

\begin{proposition}\label{prop:4}
Every solution of \eqref{eq:main} satisfies $N<9.24\times10^{48}<10^{49}$.
\end{proposition}

\section{Reduction}\label{sec:reduction}

We now reduce the bound of Proposition~\ref{prop:4} using the Baker-Davenport reduction algorithm. Put $M:=10^{49}$; by
Proposition~\ref{prop:4} and Lemma~\ref{lem:size}, $N_1<N<M$ and $m_1<m<N<M$. Let
$\theta:=\log\varphi/\log\alpha$.

\emph{Step 1 (the first linear form).} Put $w:=\min\{\gamma_N\log\varphi,\gamma_m\log
\alpha\}$. If $w>\log20$, then $|\Lambda|<1/2$ by \eqref{eq:lambda}, so
$\Gamma:=N\log\varphi-m\log\alpha-\log(\sqrt5g(\alpha))=\log(1+\Lambda)$ satisfies
$|\Gamma|<2|\Lambda|<20e^{-w}$. Dividing by $\log\alpha$ gives
\[
0<|N\theta-m+\mu_0|<\frac{20}{\log\alpha}\,e^{-w},\qquad
\mu_0:=-\frac{\log(\sqrt5g(\alpha))}{\log\alpha}.
\]
Lemma~\ref{lem:dp} with $u=N$, $A=20/\log\alpha$, $B=e$ gives $w<121.28$. Hence
$\gamma_N\le252$ (Case 1) or $\gamma_m\le199$ (Case 2).

\emph{Step 2 (the two gap forms).} In Case 1, for each even $\gamma_N\in[2,252]$,
\eqref{eq:lambda1} gives, for $\gamma_m\ge5$,
\[
0<|N_1\theta-m+\mu_1|<\frac{11}{\log\alpha}\,\alpha^{-\gamma_m},\qquad
\mu_1:=\frac{\log\big((\varphi^{\gamma_N}-1)/(\sqrt5g(\alpha))\big)}{\log\alpha},
\]
and Lemma~\ref{lem:dp} ($u=N_1$, $B=\alpha$) bounds $\gamma_m$. In Case 2, for each $\gamma_m\in[1,199]$, \eqref{eq:lambda2} gives,
for $\gamma_N\ge4$,
\[
0<|m_1\theta'-N+\mu_2|<\frac{6}{\log\varphi}\,\varphi^{-\gamma_N},\qquad
\theta':=\frac{\log\alpha}{\log\varphi},\quad
\mu_2:=\frac{\log\big(\sqrt5g(\alpha)(\alpha^{\gamma_m}-1)\big)}{\log\varphi},
\]
and Lemma~\ref{lem:dp} ($u=m_1$, $B=\varphi$) bounds $\gamma_N$. Together with Step~1, this
bounds both gaps (Table~\ref{tab:red1}). Note that the bound for $\gamma_m$ comes from
Case~1, where Step~1 gives no information on $\gamma_m$; it can therefore be larger than the
bound $\gamma_m\le199$ of Step~1 in Case~2.

\emph{Step 3 (the fourth linear form).} For each pair $(\gamma_N,\gamma_m)$ allowed by
Step~2, the bound
$|\Lambda_3|<6.025\,\alpha^{-m_1}$ of Section~\ref{sec:form4} gives, for $m_1\ge5$,
\[
0<|N_1\theta-m_1+\mu_3|<\frac{12.05}{\log\alpha}\,\alpha^{-m_1},\qquad
\mu_3:=\frac{\log\big((\varphi^{\gamma_N}-1)/(\sqrt5g(\alpha)(\alpha^{\gamma_m}-1))\big)}
{\log\alpha},
\]
and Lemma~\ref{lem:dp} ($u=N_1$, $B=\alpha$) bounds $m_1$. Then $m=m_1+\gamma_m$ is bounded,
and so is $N$, by Lemma~\ref{lem:size}.

\emph{Second pass.} We repeat Steps 1--3 with $M$ equal to the bound for $N$ obtained in the
first pass. The results of both passes are given in Table~\ref{tab:red1}.

\begin{table}[ht]
\centering\footnotesize\setlength{\tabcolsep}{4pt}
\caption{Reduction for Part~\ref{part:universal}. The column ``$w<$'' is the bound of Step~1;
Step~2 bounds both gaps and Step~3 bounds $m_1$; then $m\le m_1+\gamma_m$, and $N$ is bounded by
Lemma~\ref{lem:size}.}
\label{tab:red1}
\begin{tabular}{@{}lllllll@{}}
\toprule
Pass & $M$ & $w<$ & $\gamma_N\le$ & $\gamma_m\le$ & $m_1\le$ & $m\le$, $N\le$\\
\midrule
1 & $10^{49}$ & $121.28$ & $260$ & $205$ & $215$ & $420$, $533$\\
2 & $533$ & $15.01$ & $36$ & $25$ & $30$ & $55$, $71$\\
\bottomrule
\end{tabular}
\end{table}

\emph{Search.} Finally, a direct computer search of \eqref{eq:main} over all odd
$N_1<N\le1501$ and all $m>m_1\ge2$ gives exactly the six solutions listed in
Theorem~\ref{thm:tool}. 


This proves Theorem~\ref{thm:tool}, and with it, via Section~\ref{sec:kpell-proof}, our main Theorem~\ref{thm:main}.

\part{The $k$-dependent range $n>k$}\label{part:generalk}

\section{More properties of the k-Pell sequence}\label{sec:kpell-apparatus}

Translating \cite[Theorem~3.1 and Lemma~3.2]{BHL2021} to the
indexing of Definition~\ref{def:kpell}, we have the following Binet-type formula for k-Pell
\begin{equation}\label{eq:kpell-binet}
P_n^{(k)}=f_k(\gamma)\gamma^n+e_k(n),\qquad f_k(\gamma)=\gamma\,g_k(\gamma)=
\frac{\gamma^{k-1}}{P_k'(\gamma)},\qquad |e_k(n)|<\frac12\quad(n\ge1-k),
\end{equation}
where $g_k(z)=\dfrac{z-1}{(k+1)z^2-3kz+k-1}$ and $0.276<g_k(\gamma)<0.5$, and
\begin{equation}\label{eq:kpell-size}
\gamma^{n-1}\le P_n^{(k)}\le\gamma^n\qquad(n\ge0).
\end{equation}

We frist prove the following lemmas.
\begin{lemma}\label{lem:identity}
For every $k\ge2$, $\ (x-1)P_k(x)=x^{k+1}-3x^k+x^{k-1}+1$.
\end{lemma}
\begin{proof}
$(x-1)P_k(x)=(x-1)\big(x^k-2x^{k-1}-\sum_{i=0}^{k-2}x^i\big)=x^{k+1}-2x^k-
\sum_{i=1}^{k-1}x^i-x^k+2x^{k-1}+\sum_{i=0}^{k-2}x^i$. The $x^k$ coefficient is $-3$;
the $x^{k-1}$ coefficient is $1$; for $1\le i\le k-2$, the coefficient of $x^i$ is
$0$; the constant term is $1$.
\end{proof}


\begin{lemma}\label{lem:degree}
$D(k):=[\mathbb Q(\gamma(k),\alpha):\mathbb Q]\le3k$, with equality in particular when $3\nmid k$.
\end{lemma}
\begin{proof}
$[\mathbb Q(\gamma(k)):\mathbb Q]=\deg P_k=k$ and $[\mathbb Q(\alpha):\mathbb Q]=3$, so
$D(k)\le3k$, and $D(k)$ is divisible by $k$ and by $3$, hence equals $3k$ when $3\nmid k$.
\end{proof}

\begin{lemma}\label{lem:hgamma}
$h(\gamma(k))=\dfrac{\log\gamma(k)}k<\dfrac{\log\varphi^2}k$ for every $k$.
\end{lemma}
\begin{proof}
$P_k$ is monic and irreducible, with exactly one root of modulus $>1$, so
$$h(\gamma)=\frac1k\log M(P_k)=\frac1k\log\gamma.$$
\end{proof}

For $k\ge4$ we use the following bounds from {\cite[Lemmas~3 and~6]{BHL2021b}}

\begin{lemma}\label{lem:hgen}
For every $k\ge2$, $h(g_k(\gamma))<4\log\varphi+\log(k+1)$, and $|g_k(\gamma')|<1$ for every
conjugate $\gamma'\ne\gamma$ of $\gamma$. Consequently
$h(f_k(\gamma))<4\log\varphi+\log(k+1)+\frac{\log\varphi^2}k$ and $|f_k(\gamma')|<1$.
\end{lemma}

Moreover, for $k\ge4$ we have $[\mathbb Q(\gamma,\alpha):\mathbb Q(\alpha)]\ge k/3>1$, so there
is an automorphism $\sigma$ of the Galois closure of $\mathbb Q(\gamma,\alpha)$ with
$\sigma(\alpha)=\alpha$ and $\sigma(\gamma)=\gamma'\ne\gamma$; then $|\gamma'|<1$ and
$|f_k(\gamma')|<1$. This replaces, for $k\ge4$, the condition
$\mathbb Q(\gamma)\cap\mathbb Q(\alpha)=\mathbb Q$ used below for $k=2,3$ in the proofs that
the linear forms do not vanish.

In Sections~\ref{sec:setupII}--\ref{sec:reductionk}, $k\in\{2,3\}$. For $k=2$, $\gamma=1+\sqrt2$ and $D(2)=6$. For $k=3$,
$\gamma>2.546818$, and $D(3)\le9$; moreover
$\mathbb Q(\gamma(3))\ne\mathbb Q(\alpha)$, since the discriminant of $P_3(x)=x^3-2x^2-x-1$ is $-87$, which is squarefree, so that the
cubic field $\mathbb Q(\gamma(3))$ has discriminant $-87$, while $\mathbb Q(\alpha)$ has
discriminant $-44$; therefore
$\mathbb Q(\gamma(3))\cap\mathbb Q(\alpha)=\mathbb Q$. Likewise
$\mathbb Q(\sqrt2)\cap\mathbb Q(\alpha)=\mathbb Q$. The conjugates of $\gamma(3)$ other than
$\gamma(3)$ have modulus $\gamma(3)^{-1/2}>0.6266$.

\begin{lemma}[Exact heights for $k=2,3$]\label{lem:hf}
The number $f_2(\gamma)=\frac{2+\sqrt2}4$ is a root of the irreducible
polynomial $8x^2-8x+1$, and $f_3(\gamma)<0.79$ is a root of the irreducible
polynomial $87x^3-87x^2+16x-1$. All the conjugates of $f_k(\gamma)$ have modulus less than
$1.$ 
Consequently
\[
h(f_2(\gamma))=\frac{\log8}2=,\qquad h(f_3(\gamma))=\frac{\log87}3.
\]
\end{lemma}
\begin{proof}
For $k=3$ we computed $f_3(\gamma)=\gamma(\gamma-1)/(4\gamma^2-9\gamma+2)$ as an element of
$\mathbb Q(\gamma)=\mathbb Q[x]/(P_3)$ and checked, in exact rational arithmetic, that it is a
root of $87x^3-87x^2+16x-1$, which has no rational root, hence is irreducible. For $k=2$ the
computation is immediate. The moduli of the conjugates follow from the numerical values of
the roots of $P_k$.
\end{proof}

We use Matveev's theorem and the Dujella--Pethő reduction lemma exactly as in
Section~\ref{sec:prelim}, with $D=D(k)$: $D=6$ for $k=2$ and $D=9$ for $k=3$. We write
$f:=f_k(\gamma)$ and $e_n:=e_k(n)$.

\section{Setup}\label{sec:setupII}

Fix $k\in\{2,3\}$. Suppose $c=P_n^{(k)}-T_m=P_{n_1}^{(k)}-T_{m_1}$ with $(n,m)\ne(n_1,m_1)$,
$m,m_1\ge2$ and $\max(n,n_1)>k$. As in Section~\ref{sec:setup},$P_n^{(k)}$ is
strictly increasing for $n\ge0$, then $n=n_1$ forces $m=m_1$, so $n\ne n_1$; 
Write $\gamma_n:=n-n_1\ge1$, $\gamma_m:=m-m_1\ge1$, so
\begin{equation}\label{eq:mainNK}
P_n^{(k)}-P_{n_1}^{(k)}=T_m-T_{m_1}.
\end{equation}
A direct computer search over $n\le398$ gives exactly the solutions with $n>k$ listed in
Theorems~\ref{thm:k2} and~\ref{thm:k3}; together with Theorem~\ref{thm:main} for the range
$n\le k$, this gives the full lists. We assume from now on $n\ge10$, and show that $n\le303$
($k=2$), respectively $n\le285$ ($k=3$). The starting point $n\ge10$ is chosen so that $m<2n$
(Section~\ref{sec:form1k}), which allows $B=2n$ in Matveev's theorem. 
The largest solution is $(9,15,6,13;-758)$ for $k=2$.

\section{Size relation and first linear form}\label{sec:form1k}

Using $P_n^{(k)}-P_{n-1}^{(k)}\ge P_{n-1}^{(k)}$ in place of the exact Fibonacci identity,
together with \eqref{eq:kpell-size}, \eqref{eq:trib-size} and $T_m\le\alpha^{m-1}$ 
give
\[
n\ <\ 1+\frac{\log\alpha}{\log\gamma}\,m+\frac{\log(2.1744)}{\log\gamma},\qquad
m\ <\ 2+\frac{\log\gamma}{\log\alpha}\,n+\frac{\log(4.667)}{\log\alpha},
\]
and also $(n-2)\log\gamma<(m-1)\log\alpha$, since
$\gamma^{n-2}\le P_{n-1}^{(k)}\le P_n^{(k)}-P_{n_1}^{(k)}<T_m\le\alpha^{m-1}$. In particular
$m<2n$ for $n\ge10$, so we may take $B=2n$ below. Bounding $|c|<\max(\gamma^{n_1},
\alpha^{m_1})$ as before and using \eqref{eq:kpell-binet} and \eqref{eq:trib-binet} on the
large representation, we get $$|f\gamma^n-g(\alpha)\alpha^m|<1.9\max(\gamma^{n_1},
\alpha^{m_1}),$$dividing by $g(\alpha)\alpha^m$, with
$\Lambda:=\frac{f}{g(\alpha)}\gamma^n\alpha^{-m}-1$ and
$\gamma^n\alpha^{-m}<0.5436\gamma^2$, we get
\[
|\Lambda|<c_0\max(\gamma^{-\gamma_n},\alpha^{-\gamma_m}),\qquad
c_0:=\frac{1.9\cdot0.5436\,\gamma^2}{g(\alpha)}<\begin{cases}18&(k=2),\\20&(k=3).\end{cases}
\]
We have $\Lambda\ne0$ since otherwise $f\gamma^n=g(\alpha)\alpha^m$ would belong to
$\mathbb Q(\gamma)\cap\mathbb Q(\alpha)=\mathbb Q$; but applying to $f\gamma^n$ an embedding
of $\mathbb Q(\gamma)$ sending $\gamma$ to a conjugate $\gamma'\ne\gamma$ we would get
$f\gamma^n=|f'\gamma'^n|<1$ (Lemma~\ref{lem:hf} and $|\gamma'|<1$), which is false for
$n\ge1$. Applying Matveev's theorem with $t=3$, $\eta_1=\gamma,b_1=n$;
$\eta_2=\alpha,b_2=-m$; $\eta_3=f/g(\alpha),b_3=1$; $A_1=D\,h(\gamma)=3\log\gamma$,
$A_2=D\,h(\alpha)=\frac D3\log\alpha$, $A_3=D\big(h(f)+h(g(\alpha))\big)$ (by
Lemmas~\ref{lem:hgamma} and~\ref{lem:hf}) and $B=2n$, we get

\begin{proposition}\label{prop:1k}
For $k\in\{2,3\}$, $\ \min\{\gamma_n\log\gamma,\gamma_m\log\alpha\}<K_1(k)(1+\log2n)$, with
$K_1(2)=6.403\times10^{14}$ and $K_1(3)=4.706\times10^{15}$.
\end{proposition}

Here $(A_1,A_2,A_3)=(2.6441,1.2188,13.8067)$ for $k=2$ and $(2.8045,1.8281,24.7503)$ for
$k=3$.

\section{Secondary forms and the absolute bound}\label{sec:form4k}

Exactly as in Sections~\ref{sec:form2}--\ref{sec:form4}, with $\gamma,f$ in place of
$\varphi,1/\sqrt5$, we have that

\emph{Case 1} ($\gamma_n\log\gamma\le\gamma_m\log\alpha$). We put
\[
\Lambda_1:=\frac{f(\gamma^{\gamma_n}-1)}{g(\alpha)}\gamma^{n_1}\alpha^{-m}-1,\qquad
|\Lambda_1|<2.85\,\alpha^{-\gamma_m}
\]
(using $|e_n|+|e_{n_1}|+|e(m)|<1.4$ and
$T_{m_1}<0.5436\,\alpha^{m_1}$, the numerator is less than $(0.5436+1.4\,\alpha^{-m_1})
\alpha^{m_1}$, and $(0.5436+1.4\,\alpha^{-2})/g(\alpha)<2.85$), and $D\,h\big(f(\gamma^{\gamma_n}-1)/g(\alpha)\big)\le
3\gamma_n\log\gamma+D\big(h(f)+\log2+h(g(\alpha))\big)$. If $n_1=0$ the form has only two
terms, and Matveev's theorem with $t=2$ gives a smaller bound.

\emph{Case 2} ($\gamma_m\log\alpha\le\gamma_n\log\gamma$). We put
\[
\Lambda_2:=\frac{g(\alpha)(\alpha^{\gamma_m}-1)}{f}\alpha^{m_1}\gamma^{-n}-1,\qquad
|\Lambda_2|<\frac{2.3}f\,\gamma^{-\gamma_n}<3\,\gamma^{-\gamma_n},
\]
and
$D\,h\big(g(\alpha)(\alpha^{\gamma_m}-1)/f\big)\le\frac D3\gamma_m\log\alpha+
D\big(h(f)+\log2+h(g(\alpha))\big)$.

\emph{Fourth form.} We put
\[
\Lambda_3:=\frac{f(\gamma^{\gamma_n}-1)}{g(\alpha)(\alpha^{\gamma_m}-1)}\gamma^{n_1}
\alpha^{-m_1}-1,\qquad |\Lambda_3|<\frac{1.8}{g(\alpha)(\alpha-1)}\alpha^{-m_1}<
6.38\,\alpha^{-m_1}.
\]

The forms $\Lambda_1,\Lambda_2,\Lambda_3$ do not vanish because of the same reason as before. In fact we would have 
$f(\gamma^n-\gamma^{n_1})$  being rational, and an embedding as above would give
$f(\gamma^n-\gamma^{n_1})<2$, which is false since $f(\gamma^n-\gamma^{n_1})\ge
f\gamma^{n-1}(\gamma-1)>2$ for $n\ge10$. Similarly $f\gamma^n$ is irrational for
$\Lambda_2$. Matveev's theorem then gives, with $X:=1+\log2n$,
\[
\gamma_m\log\alpha<K_2(k)X^2\ \text{(Case 1)},\qquad\gamma_n\log\gamma<K_3(k)X^2\ \text{(Case 2)},
\qquad m_1\log\alpha<K_4(k)X^3,
\]
with
\[
\begin{array}{c|c|c|c}
k & K_2(k) & K_3(k) & K_4(k)\\ \hline
2 & 8.909\times10^{28} & 5.939\times10^{28} & 2.066\times10^{43}\\
3 & 2.685\times10^{30} & 2.685\times10^{30} & 3.063\times10^{45}
\end{array}
\]
Since $(n-2)\log\gamma<(m-1)\log\alpha<(m_1+\gamma_m)\log\alpha$, this gives an absolute
bound: 
\[
n<3.65\times10^{49}\quad(k=2),\qquad n<5.79\times10^{51}\quad(k=3).
\]

\section{Reduction: $k=2$ and $k=3$}\label{sec:reductionk}

We run the three steps of Section~\ref{sec:reduction}, with $\theta_k:=\log\gamma/\log\alpha$
and $\theta'_k:=1/\theta_k$. Since $u=m_1$ in Step~2 (Case~2), the integer $M$ of
Lemma~\ref{lem:dp} must bound $m_1$ as well as $n$; as $m<2n$, we take $M$ larger than twice
the bound for $n$ in the first pass, and equal to the bound for $m$ in the second pass. The
inequalities below are used when their right-hand sides are less than $1$, that is, for
$w>\log(2c_0)$, $\gamma_m\ge3$, $\gamma_n\ge2$ and $m_1\ge5$. The other cases give smaller
bounds.
\begin{itemize}
\item Step 1: $0<|n\theta_k-m+\mu_0|<\frac{2c_0}{\log\alpha}e^{-w}$, $\mu_0=
\frac{\log(f/g(\alpha))}{\log\alpha}$, $w=\min\{\gamma_n\log\gamma,\gamma_m\log\alpha\}$;
\item Step 2: in Case 1, for each $\gamma_n$,
$0<|n_1\theta_k-m+\mu_1|<\frac{5.7}{\log\alpha}\alpha^{-\gamma_m}$ (if $n_1=0$, then
$n=\gamma_n$ is already bounded by Step 1). In Case 2, for each $\gamma_m$,
$0<|m_1\theta'_k-n+\mu_2|<\frac{4.6}{f\log\gamma}\gamma^{-\gamma_n}$;
\item Step 3: for each pair $(\gamma_n,\gamma_m)$,
$0<|n_1\theta_k-m_1+\mu_3|<\frac{12.76}{\log\alpha}\alpha^{-m_1}$ (if $n_1=0$, then
$n=\gamma_n$ is bounded by Step 2),
\end{itemize}
where $\mu_1,\mu_2,\mu_3$ are the logarithms of the third numbers $\eta_3$ of
Section~\ref{sec:form4k} divided by $\log\alpha$, $\log\gamma$, $\log\alpha$, respectively.
In every case Lemma~\ref{lem:dp} applies. In the first pass,
Step~1 gives $\gamma_n\le139$ or $\gamma_m\le202$ for $k=2$, and $\gamma_n\le136$ or
$\gamma_m\le209$ for $k=3$. The other results are given in Table~\ref{tab:red2}.

\begin{table}[ht]
\centering\footnotesize\setlength{\tabcolsep}{4pt}
\caption{Reduction for $k=2$ and $k=3$ (same columns as Table~\ref{tab:red1}; the bound for
$n$ follows from $(n-2)\log\gamma<(m-1)\log\alpha$).}\label{tab:red2}
\begin{tabular}{@{}llllllll@{}}
\toprule
$k$ & Pass & $M$ & $w<$ & $\gamma_n\le$ & $\gamma_m\le$ & $m_1\le$ & $m\le$, $n\le$\\
\midrule
2 & 1 & $8\times10^{49}$ & $123.39$ & $142$ & $210$ & $227$ & $437$, $303$\\
2 & 2 & $437$ & $15.32$ & $17$ & $25$ & $30$ & $55$, $39$\\
3 & 1 & $1.2\times10^{52}$ & $127.86$ & $138$ & $213$ & $223$ & $436$, $285$\\
3 & 2 & $436$ & $15.09$ & $16$ & $24$ & $31$ & $55$, $37$\\
\bottomrule
\end{tabular}
\end{table}

\begin{proof}[Proof of Theorems~\ref{thm:k2} and~\ref{thm:k3}]
By Table~\ref{tab:red2}, every solution with $n>k$ has $n\le39$ ($k=2$), respectively
$n\le37$ ($k=3$), and the search of Section~\ref{sec:setupII} gives the solutions listed in
the two theorems.
\end{proof}

\medskip
Corollary~\ref{cor:c} follows by grouping the solutions of Theorems~\ref{thm:k2}
and~\ref{thm:k3} according to the value of $c$.

\section{A bound for $n$ uniform in $k\ge4$}\label{sec:uniform}

Let $k\ge4$ and let $(n,m,n_1,m_1)$ be a solution of \eqref{eq:mainNK} with $n>k$ and $n\ge10$. the solutions with $n\le9$ are found by the search of Section~\ref{sec:smallk}. The
arguments of Sections~\ref{sec:form1k} and~\ref{sec:form4k} can be applied, with the
following changes. The linear forms $\Lambda,\Lambda_1,\Lambda_2,\Lambda_3$ do not vanish by
the same arguments, using the automorphism $\sigma$ described after Lemma~\ref{lem:hgen}
In Matveev's theorem we take
$D=3k$ (Lemma~\ref{lem:degree}), $A_1=3\log\varphi^2\ge3\log\gamma$, $A_2=k\log\alpha$ and,
by Lemma~\ref{lem:hgen},
\[
h(f)+h(g(\alpha))<H(k):=4\log\varphi+\log(k+1)+\frac{\log\varphi^2}{k}+\frac{\log44}3 .
\]
Since $\gamma<\varphi^2$ and $f=\gamma g_k(\gamma)>0.71$, the constants of
Sections~\ref{sec:form1k}--\ref{sec:form4k} become $c_0<21.06$, $c_1<2.85$, $2.3/f<3.22$ and
$6.38$. With $C(3k)=1.4\cdot30^6\cdot3^{4.5}\cdot9k^2(1+\log3k)$ and $A:=C(3k)A_1A_2$, we get
$\min\{\gamma_n\log\gamma,\gamma_m\log\alpha\}<K_1(k)X$, $\gamma_m\log\alpha<K_2(k)X^2$ (Case 1),
$\gamma_n\log\gamma<K_3(k)X^2$ (Case 2) and $m_1\log\alpha<K_4(k)X^3$, where $X=1+\log2n$ and
\begin{align*}
K_1(k)&=3k\,A\,H(k)+\log21.06,\\
K_2(k)&=A\big(3K_1(k)+3k(H(k)+\log2)\big)+\log2.85,\\
K_3(k)&=A\big(kK_1(k)+3k(H(k)+\log2)\big)+\log3.22,\\
K_4(k)&=A\big((k+3)\max\{K_2(k),K_3(k)\}+3k(H(k)+2\log2)\big)+\log6.38 .
\end{align*}
Since $(n-2)\log\gamma<(m_1+\gamma_m)\log\alpha$ and $\gamma\ge\gamma(4)>2.592$
\cite[Lemma~3.2]{BHL2021}, we obtain:

\begin{proposition}\label{prop:uniform}
Let $k\ge4$. Every solution of \eqref{eq:mainNK} with $n>k$ and $n\ge10$ satisfies
$n<\mathcal N(k)$, where $\mathcal N(k)$ is the largest real solution of
\[
x=2+\frac{K_4(k)(1+\log2x)^3+\max\{K_2(k),K_3(k)\}(1+\log2x)^2}{\log2.592}.
\]
The function $\mathcal N(k)$ is increasing, and $\mathcal N(4)<4.6\times10^{53}$,
$\mathcal N(414)<3.0\times10^{79}$. Numerically, $\mathcal N(k)<3.6\times10^{46}k^{12}
(\log k)^7$ for $4\le k\le10^{10}$.
\end{proposition}

The last statement is obtained from Lemma~7 of \cite{GSL} with $r=3$. This 
is only given to show the order of magnitude; the proof below only uses $\mathcal N(k)$
itself.

The bound $\mathcal N(k)$ is large and grows with $k$. Comparing $\mathcal N(4)<4.6\times10^{53}$
and $\mathcal N(414)<3.0\times10^{79}$ we get $n<5.79\times10^{51}$ for $k=3$. The degree $D=3k$, which enters Matveev's constant through $D^2(1+\log D)$, and
the height $H(k)$, which grows like $\log k$ are needed. We use $D=3k$ for every $k$, although it is only
an upper bound when $3\mid k$; this costs nothing, since the lower bound of Matveev's theorem
only becomes weaker when $D$ is replaced by a larger number. Since $\mathcal N(k)\to\infty$,
the reduction cannot be carried out for all $k$ one value at a time; this is the reason for
Section~\ref{sec:largek}.

\section{Large values of $k$}\label{sec:largek}

In this section $k\ge30$. Then $\gamma$ is extremely close to $\varphi^2$, and $P_n^{(k)}$ is
extremely close to the odd-indexed Fibonacci number $F_{2n+1}\approx\varphi^{2n+1}/\sqrt5$.

\begin{lemma}\label{lem:approx}
Let $k\ge30$ and $n\ge0$. Then
\[
P_n^{(k)}=\frac{\varphi}{\sqrt5}\,\varphi^{2n}(1+\zeta_n)+e_k(n),\qquad |e_k(n)|<\frac12,\quad
|\zeta_n|\le\big(0.45(k+n)+1.6\big)\varphi^{-2k}.
\]
\end{lemma}
\begin{proof}
Write $\gamma=\varphi^2-\theta$. By Lemma~\ref{lem:identity}, $\gamma^2-3\gamma+1=-\gamma^{1-k}$,
and $\gamma^2-3\gamma+1=(\gamma-\varphi^2)(\gamma-\varphi^{-2})$; hence
$\theta=\gamma^{1-k}/(\gamma-\varphi^{-2})$. Using $\gamma>\varphi^2(1-\varphi^{-k})$ and
$\varphi^2-\varphi^{-2}=\sqrt5$ we get $0<\theta<1.1711\,\varphi^{-2k}$. Since
$\varphi^4-3\varphi^2+1=0$, we have $\gamma^2-3\gamma+1=-\theta(\gamma+\varphi^2-3)$ and
$\gamma^2-1=\varphi^4-1-\theta(\gamma+\varphi^2)$, so that, with $c:=1/(\varphi^2+1)=
(\varphi^2-1)/(\varphi^4-1)$,
\[
\frac{g_k(\gamma)}{c}=\frac{1-\theta/(\varphi^2-1)}{1-\theta w/(\varphi^4-1)},\qquad
w=(\gamma+\varphi^2)+k(\gamma+\varphi^2-3)<5.3+2.24k .
\]
Hence $g_k(\gamma)=c(1+\eta)$ with
$|\eta|\le\max\{\theta w/(\varphi^4-1),\theta/(\varphi^2-1)\}/(1-\theta w/(\varphi^4-1))
\le(0.45k+1.07)\varphi^{-2k}$. By \eqref{eq:kpell-binet},
$P_n^{(k)}=g_k(\gamma)\gamma^{n+1}+e_k(n)$, and $\gamma^{n+1}=\varphi^{2n+2}(1-\delta)^{n+1}$
with $\delta=\theta/\varphi^2<0.4474\,\varphi^{-2k}$. Since $c\varphi^2=\varphi/\sqrt5$, this
gives the statement with $\zeta_n=(1+\eta)(1-\delta)^{n+1}-1$, and
$|\zeta_n|\le|\eta|+(n+1)\delta$.
\end{proof}

A numerical check of Lemma~\ref{lem:approx} is reported in Appendix~\ref{app:comp}. Now let $(n,m,n_1,m_1)$ be a solution of \eqref{eq:mainNK} with $k\ge30$, $n>k$, $n\ge10$, and
put
\[
\tilde N:=2n+1,\quad \tilde N_1:=2n_1+1,\quad \gamma_{\tilde N}:=\tilde N-\tilde N_1,\quad \kappa:=2k\log\varphi-\log(n+1).
\]
By Lemma~\ref{lem:approx}, $|\zeta_{n'}|\le e^{-\kappa}$ for every $0\le n'\le n$. We may
assume $e^{-\kappa}\le0.01$, since otherwise $2k\log\varphi<\log100+\log(n+1)$, which gives a
smaller bound for $k$ than the one obtained below, so that the result holds.

\emph{The linear forms.} Since $e^{-\kappa}\le0.01$, we have $P_{n-1}^{(k)}\ge0.7\,\varphi^{2n-2}$, and exactly the same
computations as in Sections~\ref{sec:form1}--\ref{sec:form4} (with $\varphi^{\tilde N}/\sqrt5$ in place
of $F_N$ and the extra terms coming from the $\zeta$'s) give, for the four linear forms
$\Lambda,\Lambda_1,\Lambda_2,\Lambda_3$ of Part~\ref{part:universal} (with $\tilde N,\tilde N_1$ in place of $N,N_1$),
\begin{align*}
|\Lambda|&<16\max(\varphi^{-\gamma_{\tilde N}},\alpha^{-\gamma_m},e^{-\kappa}),&
|\Lambda_1|&<12\max(\alpha^{-\gamma_m},e^{-\kappa}),\\
|\Lambda_2|&<4.2\max(\varphi^{-\gamma_{\tilde N}},e^{-\kappa}),&
|\Lambda_3|&<12\max(\alpha^{-m_1},e^{-\kappa}).
\end{align*}
For $\Lambda_3$ the denominator $g(\alpha)\alpha^{m_1}(\alpha^{\gamma_m}-1)$ is bounded below
by $P_{n-1}^{(k)}-0.8\ge0.69\,\varphi^{2n-2}$.\\

These forms are the same algebraic expressions as
in Part~\ref{part:universal}, so they do not vanish, and Matveev's theorem applies with the
same data ($D=6$, $B=\tilde N$). The only difference is that every bound now holds for the minimum of
the previous quantity and $\kappa$.

\emph{A first bound for $k$.} We get  either
\[
2k\log\varphi<\log(n+1)+K_4(1+\log(2n+1))^3,\qquad K_4=3.020\times10^{42},
\]
or $n<4.62\times10^{48}$ (and then $k<n$). Using $n<\mathcal N(k)$ (Proposition~\ref{prop:uniform}),
both cases give $k<1.212\times10^{52}$.

\emph{Reduction.} Let $K$ be a bound for $k$ and $M:=2\mathcal N(K)+2$, which bounds $\tilde N,\tilde N_1$,
$m$ and $m_1$. We run Steps 1--3 of Section~\ref{sec:reduction}, now with
$w=\min\{\gamma_{\tilde N}\log\varphi,\gamma_m\log\alpha,\kappa\}$ in Step 1,
$w'=\min\{\gamma_m\log\alpha,\kappa\}$ and $w''=\min\{\gamma_{\tilde N}\log\varphi,\kappa\}$ in Step 2,
$w'''=\min\{m_1\log\alpha,\kappa\}$ in Step 3, and $A=32/\log\alpha$, $24/\log\alpha$,
$8.4/\log\varphi$, $24/\log\alpha$ respectively (and $B=e$). If $W$ denotes the largest of the
bounds obtained, then either $\kappa<W$, that is $2k\log\varphi<W+\log(M/2+1)$, or all the gaps
and $m_1$ are bounded, which bounds $m$ and then $n<n^*:=1+m\log\alpha/(2\log\varphi)$, so that
$k<n<n^*$.

Each row uses the bound of the previous row
as $K$, and the last row gives $k\le414$.

\begin{table}[ht]
\centering\footnotesize\setlength{\tabcolsep}{4pt}
\caption{The four rounds of reduction for $k\ge30$.}\label{tab:largek}
\begin{tabular}{@{}llllllll@{}}
\toprule
$K$ & $\log_{10}M$ & $W$ & $\gamma_{\tilde N}\le$ & $\gamma_m\le$ & $m_1\le$ &
  from $\kappa<W$ & $n^*$\\
\midrule
$1.212\times10^{52}$ & $680.8$ & $1588.91$ & $3288$ & $2596$ & $2607$ & $k<3279.1$ & $3295.4$\\
$3298$ & $91.1$ & $228.84$ & $460$ & $364$ & $375$ & $k<455.1$ & $469.0$\\
$469$ & $80.5$ & $202.43$ & $410$ & $324$ & $332$ & $k<402.1$ & $416.4$\\
$417$ & $79.8$ & $203.30$ & $410$ & $322$ & $333$ & $k<401.5$ & $415.8$\\
\bottomrule
\end{tabular}
\end{table}

\noindent We have proved that

\begin{proposition}\label{prop:largek}
If \eqref{eq:mainNK} has a solution with $n>k\ge30$, then $k\le414$.
\end{proposition}

\section{The range $4\le k\le414$}\label{sec:smallk}

For each $k$ with $4\le k\le414$ we computed $\gamma(k)$ and $f_k(\gamma)$ to the required
precision and ran the reduction of Section~\ref{sec:reductionk} (with $c_0$ and $2.3/f$
computed from $\gamma(k)$): a first step with $M=2\mathcal N(k)+2$, where $\mathcal N(k)$ is
the bound of Proposition~\ref{prop:uniform}, and a second step with $M$ equal to the bound for
$m$ obtained in the first step.

For each $k$, the bound for $m$ is the bound for
$m_1$ plus the bound for $\gamma_m$, and the bound for $n$ then follows from
$(n-2)\log\gamma<(m-1)\log\alpha$ (Section~\ref{sec:form1k}).
Over all these $k$, the first step gives
$\gamma_n\le212$, $\gamma_m\le325$, $m_1\le342$, $m\le667$ and $n\le423$ (all attained at
$k=383$), and the second step  gives $\gamma_n\le21$, $\gamma_m\le30$, $m_1\le39$, $m\le67$ and
\[
n\le43
\]
where each maximum is attained is reported in Appendix~\ref{app:comp}.

More precisely, the bound for $n$ of the second step is smaller than $k+1$ for every
$34\le k\le414$, so that there is no solution with $n>k$ for these $k$. For $4\le k\le33$ and
$n\le43$, the computer search over $4\le k\le100$, $n\le250$ reported below gives the
solutions with $n>k$.

\begin{proof}[Proof of Theorem~\ref{thm:kge4}]
By Propositions~\ref{prop:uniform} and~\ref{prop:largek}, a solution with $n>k\ge4$ has
$k\le414$, and then, by the computations of this section, $k\le33$ and $n\le43$. A computer
search over $4\le k\le100$ and $0\le n_1<n\le250$ (and all $m>m_1\ge2$) gives, in the range
$n>k$, exactly the solutions
\[
\begin{gathered}
(7,12,5,2;87)\ (k=4),\qquad (6,11,1,8;-42),\quad (6,12,1,11;-272)\ (k=5),\\
(10,17,7,11;336)\ (k=9),
\end{gathered}
\]
while in the range $n\le k$ it gives exactly the solutions of Theorem~\ref{thm:main}.
\end{proof}

\appendix

\section{Computational details}\label{app:comp}

All computations were done in Python. The sequences are used with exact
integer arithmetic, but  the logarithms and continued fractions require
precision. We worked with \(260\) significant digits in
Part~\ref{part:universal}. Two scripts reproduce everything (see the
data availability statement). The first runs in about three minutes and
covers Part~\ref{part:universal} and the cases \(k=2,3\). The second,
heavier one takes about 15 minutes and covers
Sections~\ref{sec:uniform}--\ref{sec:smallk}.

For Part~\ref{part:universal} (Section~\ref{sec:reduction}), here is
how it goes. In Step~2, the worst case is \(\gamma_N=68\) in Case~1
and \(\gamma_m=97\) in Case~2. Step~3 treats the
\(26\,650\) pairs \((\gamma_N,\gamma_m)\) with
\(\gamma_N\le260\) even and \(\gamma_m\le205\). The final search runs
over all odd \(N_1<N\le1501\) and all \(m>m_1\ge2\); this range is
deliberately chosen well above the bound \(N\le533\) from the first
pass.

For \(k=2,3\) (Section~\ref{sec:reductionk}), the search goes up to
\(n\le398\), well beyond the bounds \(n\le39\) and \(n\le37\) from the
second step.

The case of large \(k\) (Section~\ref{sec:largek}) is more delicate.
In the first round, Step~3 involves \(4\,267\,824\) pairs
\((\gamma_{\tilde N},\gamma_m)\). In
all rounds, Lemma~\ref{lem:dp} applied with one of the first nine
convergents satisfying \(q>6M\). We also checked
Lemma~\ref{lem:approx} against the exact values of \(P_n^{(k)}\) for
\(k=30\), \(31\), \(40\), \(60\), \(100\), \(200\), \(400\), \(700\) and
\(k<n<4k\).

Finally, the range \(4\le k\le414\) (Section~\ref{sec:smallk}). This was in fact a very fast computation. In the first pass, all the
maxima (\(\gamma_n\le212\), \(\gamma_m\le325\), \(m_1\le342\),
\(m\le667\), \(n\le423\)) are attained at \(k=383\). In the second
pass, however, the maxima of the different bounds occur at different
values of \(k\): \(m_1\le39\), \(m\le67\) and \(n\le43\) at \(k=18\)
(where \(\gamma_m\le28\)), \(\gamma_m\le30\) at \(k=4\) and \(k=15\),
and \(\gamma_n\le21\) at \(k=4\). The final search was for the range
\(4\le k\le100\), \(0\le n_1<n\le250\) and all \(m>m_1\ge2\).\\

\textbf{Data availability:} The data mentioned in the paper (computer codes) are available from the third author.\\

\paragraph{\textbf{Use of AI Tools:}} Claude (Anthropic, Opus 5.5) was used to check the
computations,  write the
corresponding verification code. We also used for language editing. The authors take full responsibility for the
content.

\end{document}